\documentclass[a4paper]{amsart}
\usepackage{amsmath,amssymb,amsthm,amscd}
\usepackage{enumerate,bm,color,here,url,hyperref}

\usepackage{graphicx}
\graphicspath{{./figure/}}

\theoremstyle{plain}
	\newtheorem{thm}{Theorem}[section]
	\newtheorem{lem}[thm]{Lemma}
	\newtheorem{prop}[thm]{Proposition}
	\newtheorem{cor}[thm]{Corollary}

\theoremstyle{definition}
	
\theoremstyle{remark}
	\newtheorem{rem}[thm]{Remark}

\DeclareMathOperator{\Diff}{Diff}

\DeclareMathOperator{\Isom}{Isom}

\newcommand{\bbZ}{\mathbb{Z}}
\newcommand{\bbR}{\mathbb{R}}
\newcommand{\bbC}{\mathbb{C}}

\newcommand{\bbP}{\mathbb{P}}

\begin{document}

\title{Uniqueness of free 2-periodicities of links}

\author{Ken'ichi Yoshida}
\address{International Institute for Sustainability with Knotted Chiral Meta Matter (WPI-SKCM$^2$), Hiroshima University, 1-3-1 Kagamiyama, Higashi-Hiroshima, Hiroshima 739-8531, Japan}
\email{kncysd@hiroshima-u.ac.jp}
\date{}

\begin{abstract}
We show that if two links in the real projective 3-space $\bbR \bbP^{3}$ have isotopic preimages in the 3-sphere $S^{3}$ by the double covering map, 
then they are themselves isotopic in $\bbR \bbP^{3}$. 
\end{abstract}

\maketitle

\section{Introduction}
\label{section:intro}

Links in the real projective 3-space $\bbR \bbP^{3}$ (also known as projective links) have been investigated in various ways, 
often using diagrams, as in \cite{Drobotukhina90, MN23, MRR23}. 
More generally, links in lens spaces have been considered \cite{BGH08, CMM13}. 
For a link $L$ in a lens space, 
the preimage (also called the lift) of $L$ by the universal covering map of the lens space 
is a freely periodic link in the 3-sphere $S^{3}$ according to Hartley~\cite{Hartley81}. 
Free periodicities of links in $S^{3}$ are quite restricted, especially for knots. 
For instance, Sakuma~\cite{Sakuma86} and Boileau and Flapan~\cite{BF87} independently proved that 
a free $n$-periodicity of a non-trivial prime knot in $S^{3}$ is unique for each $n \geq 2$. 
More precisely, for a non-trivial prime knot $K$ in $S^{3}$ and the group $\Diff^{*}(S^{3}, K)$ consisting of the diffeomorphisms of the pair $(S^{3}, K)$ which preserve the orientations of $S^{3}$ and $K$, 
if $f, g \in \Diff^{*} (S^{3}, K)$ have no fixed points and satisfy $f^{n} = g^{n} = \mathrm{id}$, 
then $f$ is conjugate in $\Diff^{*}(S^{3}, K)$ to a power of $g$. 
Equivalently, if $K_{0} \subset L(n, q_{0})$ and $K_{1} \subset L(n, q_{1})$ are knots in lens spaces, 
and their premiages in $S^{3}$ are isotopic knots, 
then the pairs $(L(n, q_{0}), K_{0})$ and $(L(n, q_{1}), K_{1})$ are diffeomorphic. 
Several constraints for freely periodic knots in $S^{3}$ are known \cite{BR23, BR24, Chbili03, Chbili24, Hartley81, Nozaki18}. 
Compatibility of free periodicities and alternating diagrams has also been considered \cite{Boyle21, CH23}. 
Furthermore, the finite group actions on knots in $S^{3}$ were classified in \cite{BRW23}.

In this paper, we consider links in $\bbR \bbP^{3}$ and their preimages in $S^{3}$. 
Every link is assumed to be tame. 
Let $\pi \colon S^{3} \to \bbR \bbP^{3}$ denote the universal covering map, which is a double cover. 
Two links $L_{0}$ and $L_{1}$ in $\bbR \bbP^{3}$ are called \emph{isotopic} 
if there is an ambient isotopy of $\bbR \bbP^{3}$ that maps $L_{0}$ to $L_{1}$. 
Then their preimages by the covering map $\pi$ are also isotopic in $S^{3}$ 
since the isotopy is lifted. 
The main theorem asserts that the converse also holds. 

\begin{thm}
\label{thm:main}
Let $L_{0}$ and $L_{1}$ be links in $\bbR \bbP^{3}$. 
Suppose that the links $\widetilde{L_{0}} = \pi^{-1}(L_{0})$ and $\widetilde{L_{1}} = \pi^{-1}(L_{1})$ in $S^{3}$ are isotopic. 
Then $L_{0}$ and $L_{1}$ are isotopic. 
\end{thm}

In other words, a free 2-periodicity of a link in $S^{3}$ is unique. 
In the case that $\widetilde{L_{0}}$ and $\widetilde{L_{1}}$ are prime knots, 
the assertion follows from the uniqueness of free periodicities by \cite{BF87, Sakuma86}. 
Note that the supposed isotopy between $\widetilde{L_{0}}$ and $\widetilde{L_{1}}$ may not preserve the free 2-periodicity.

Two non-isotopic links in the lens space $L(4,1)$ 
may have isotopic preimages in $S^{3}$ \cite{Manfredi14}. 
More strongly, there are infinitely many pairs of such hyperbolic links \cite{Yoshida25}. 
In other words, a free 4-periodicity of a link in $S^{3}$ is not unique in general. 
In contrast, 
Kotorii, Mahmoudi, Matsumoto, and the author \cite{KMMY25} showed that 
in the case that $X$ is the solid torus $S^{1} \times D^{2}$, the thickened torus $T^{2} \times I$, or the 3-torus $T^{3}$, 
if two links in $X$ have isotopic preimages in a finite cover of $X$, 
then they are isotopic.

We will prove the main theorem by a similar argument to that in \cite{KMMY25}. 
In Section~\ref{section:hyp}, 
we will show the main theorem for the hyperbolic links 
using the Mostow rigidity theorem and the resolution of the spherical space form conjecture. 
In Section~\ref{section:seifert}, 
we will show the main theorem for the links whose complements are Seifert fibered. 
We will explicitly describe these links, 
which are generalizations of the torus links. 
In Section~\ref{section:jsj}, 
we will show the main theorem for all links. 
By the prime and JSJ decompositions of the complement, 
we reduce to the above cases and a result for links in the solid torus \cite[Theorem 4.5]{KMMY25}. 
Then we need the notion of an outermost JSJ piece as a starting point. 
We remark that the main theorem holds for the oriented links in $\bbR \bbP^{3}$. 
This can be shown in the same manner 
by being careful with the orientation.

\section{Hyperbolic links}
\label{section:hyp}

In this section, we show that a free 2-periodicity of hyperbolic links in $S^{3}$ is unique. 
For this purpose, we need to use the now proven spherical space form conjecture as follows. 
This is a consequence of the geometrization of 3-manifolds and 3-orbifolds 
and was explicitly shown by Dinkelbach and Leeb~\cite{DL09} using equivariant Ricci flow with surgery. 

\begin{thm}[Spherical space form conjecture]
\label{thm:spherical}
A smooth action of a finite group on $S^{3}$ is smoothly conjugate to an isometric action. 
In other words, $S^{3}$ admits a spherical metric in which the action is isometric. 
\end{thm}

We consider a condition in a general setting. 
Let $L_{0}$ and $L_{1}$ be hyperbolic links in a 3-manifold $X$. 
In other words, the complement $M_{i}$ of $L_{i}$ for each $i = 0,1$ admits a hyperbolic metric of finite volume. 
Let $\pi \colon \widetilde{X} \to X$ be a finite regular covering. 
Suppose that the preimages $\widetilde{L_{0}} = \pi^{-1}(L_{0})$ and $\widetilde{L_{1}} = \pi^{-1}(L_{1})$ are isotopic in $\widetilde{X}$. 
Then the complement $\widetilde{M_{i}}$ of $\widetilde{L_{i}}$ is a finite cover of $M_{i}$ 
and is endowed with the lifted hyperbolic metric. 
The isotopy between $\widetilde{L_{0}}$ and $\widetilde{L_{1}}$ induces a homeomorphism $f \colon \widetilde{M_{0}} \to \widetilde{M_{1}}$, 
which is isotopic to an isometry $\tilde{\iota} \colon \widetilde{M_{0}} \to \widetilde{M_{1}}$ by the Mostow rigidity theorem \cite{Mostow73, Prasad73} and the fact that homotopic homeomorphisms between Haken 3-manifolds are isotopic \cite{Waldhausen68}. 
The finite group $G = \pi_{1}(X) / \pi_{1}(\widetilde{X})$ acts on $\widetilde{X}$ by the deck transformations. 
This action induces free isometric actions $\alpha_{0}$ and $\alpha_{1}$ of $G$ on the hyperbolic 3-manifolds $\widetilde{M_{0}}$ and $\widetilde{M_{1}}$. 
The isometry $\tilde{\iota} \colon \widetilde{M_{0}} \to \widetilde{M_{1}}$ induces the $G$-action $\tilde{\iota}_{*}(\alpha_{0})$ on $\widetilde{M_{1}}$ from the $G$-action $\alpha_{0}$ on $\widetilde{M_{0}}$. 

\begin{lem}[\cite{KMMY25} Lemma 4.2]
\label{lem:action}
If the $G$-actions of $\tilde{\iota}_{*}(\alpha_{0})$ and $\alpha_{1}$ on $\widetilde{M_{1}}$ coincide, 
then the pairs $(X, L_{0})$ and $(X, L_{1})$ are homeomorphic. 
Furthermore, suppose that if any self-homeomorphism $f$ of $X$ lifts to a self-homeomorphism $\tilde{f}$ of $\widetilde{X}$ isotopic to the identity, 
then $f$ is isotopic to the identity. 
In this case, $L_{0}$ and $L_{1}$ are isotopic in $X$. 
\end{lem}

We apply Lemma~\ref{lem:action} to the case that $X = \bbR \bbP^{3}$. 
It is sufficient to show that the two $\bbZ / 2\bbZ$-actions coincide 
by the following fact shown in \cite{Bonahon83, HR85} (see also \cite{CM18}). 

\begin{lem}
\label{lem:rp3-mcg}
Any orientation-preserving self-homeomorphism of $\bbR \bbP^{3}$ is isotopic to the identity. 
Consequently, 
two links $L_{0}$ and $L_{1}$ in $\bbR \bbP^{3}$ are isotopic 
if and only if there is an orientation-preserving homeomorphism between $(\bbR \bbP^{3}, L_{0})$ and $(\bbR \bbP^{3}, L_{1})$. 
\end{lem}

The following elementary fact is crucial for the uniqueness of free 2-periodicities. 

\begin{lem}
\label{lem:invol}
A free isometric $\bbZ / 2\bbZ$-action on $S^{3}$ is unique. 
It is generated by the antipodal involution. 
\end{lem}
\begin{proof}
Suppose that $A \in \mathrm{O}(4)$ acts on $S^{3}$ freely and $A^{2} = I$, 
where $I \in \mathrm{O}(4)$ is the identity matrix. 
Since $A^{2} = I$, an eigenvalue of $A$ is 1 or $-1$. 
Since $A$ acts on $S^{3}$ freely, 1 is not an eigenvalue of $A$. 
Since $-1$ is the only eigenvalue of $A$ and $A^{2} = I$, a Jordan normal form of $A$ is $-I$. 
Since $A$ is in the center of $\mathrm{GL}(4, \bbC)$, we have $A = -I$. 
\end{proof}

\begin{thm}
\label{thm:hyp-main}
Let $L_{0}$ and $L_{1}$ be hyperbolic links in $\bbR \bbP^{3}$. 
Suppose that the links $\widetilde{L_{0}} = \pi^{-1}(L_{0})$ and $\widetilde{L_{1}} = \pi^{-1}(L_{1})$ in $S^{3}$ are isotopic. 
Then $L_{0}$ and $L_{1}$ are isotopic. 
\end{thm}
\begin{proof}
In the above setting, let $X = \bbR \bbP^{3}$ and $G = \bbZ / 2\bbZ$. 
We consider the two subgroups $G_{0} = \tilde{\iota}_{*}(\alpha_{0})(G)$ and $G_{1} = \alpha_{1}(G)$ of the finite group $\Isom(\widetilde{M_{1}})$. 
Let $\overline{G}$ denote the subgroup of $\Isom(\widetilde{M_{1}})$ generated by $G_{0}$ and $G_{1}$. 
Since the elements of $G_{0}$ and $G_{1}$ preserves the meridians of $\widetilde{L_{1}}$, 
so does the elements of $\overline{G}$. 
Hence the action of the finite group $\overline{G}$ on $\widetilde{M_{1}}$ extends to a smooth action on $S^{3}$. 
We may assume that this action is isometric by Theorem~\ref{thm:spherical}. 
Since $G_{0}$ and $G_{1}$ acts freely and isometrically, 
the subgroups $G_{0}$ and $G_{1}$ coincide by Lemma~\ref{lem:invol}. 
Lemmas~\ref{lem:action} and \ref{lem:rp3-mcg} then imply that $L_{0}$ and $L_{1}$ are isotopic. 
\end{proof}

\begin{rem}
In the above proof, we obtain that the finite group $\overline{G}$ is actually $\bbZ / 2\bbZ$ 
after we obtain a (not necessarily free) isometric action of $\overline{G}$ on $S^{3}$. 
Hence we need the resolution of the spherical space form conjecture in full generality, 
as well as the case of a free involution shown by Livesay~\cite{Livesay60}. 
\end{rem}

\section{Links whose complements are Seifert fibered}
\label{section:seifert}

We classify the links in $\bbR \bbP^{3}$ whose complements are Seifert fibered. 
After that, we consider their preimages in $S^{3}$. 
The links in $S^{3}$ whose complements are Seifert fibered were classified in \cite{Budney06}.

We first prepare notations for generalization of the torus links. 
Fix genus-1 Heegaard splittings of $S^{3}$ and $\bbR \bbP^{3}$ obtained by gluing two solid tori $H_{0}$ and $H_{1}$ along the boundary $\Sigma = \partial H_{0} = \partial H_{1}$. 
Fix an oriented logitude $\lambda$ and meridian $\mu$ on $\Sigma$ as follows: 
$\mu$ bounds a disk in $H_{0}$ for $S^{3}$ and $\bbR \bbP^{3}$, 
$\lambda$ bounds a disk in $H_{1}$ for $S^{3}$, and 
$2 \lambda + \mu$ bounds a disk in $H_{1}$ for $\bbR \bbP^{3}$. 
We denote the corresponding generators of $H_{1}(\Sigma, \bbZ) \cong \bbZ^{2}$ also by $\lambda$ and $\mu$. 

For $(p, q) \in \bbZ^{2}$ and $n \in \{ 0,1,2 \}$, 
we define the links $T_{S^{3}}(p, q; n) \subset S^{3}$ and $T_{\bbR \bbP^{3}}(p, q; n) \subset \bbR \bbP^{3}$ as follows. 
If $(p, q) \neq (0,0)$, 
let $d = \gcd(p,q)$, $p' = p/d$, $q' = q/d$, 
and let $T_{*}(p, q; 0)$ for $* = S^{3}$ or $\bbR \bbP^{3}$ denote the union of $d$ parallel copies of the simple closed curve on $\Sigma$ 
representing $p' \lambda + q' \mu \in H_{1}(\Sigma, \bbZ)$. 
Conventionally, let $T_{*}(0, 0; 0)$ denote the empty set. 
Let $T_{*}(p, q; 1)$ denote the union of $T_{*}(p, q; 0)$ and the core of $H_{0}$. 
Let $T_{*}(p, q; 2)$ denote the union of $T_{*}(p, q; 0)$ and the cores of $H_{0}$ and $H_{1}$. 
Note that $T_{S^{3}}(p, q; 0)$ is the $(p,q)$-torus link.

The triple $(p, q, n)$ is not unique for the isotopy class of $T_{S^{3}}(p, q; n)$. 
Indeed, the following relations hold, 
where $L_{0} \cong L_{1}$ indicates that they are isotopic. 
\begin{enumerate}
\item $T_{S^{3}}(p, q; n) = T_{S^{3}}(-p, -q; n)$. 
\item $T_{S^{3}}(p, q; n) \cong T_{S^{3}}(q, p; n)$ for $n=0,2$. 
\item If $p > 0$ divides $q$, then $T_{S^{3}}(p, q; 0) \cong T_{S^{3}}(p-1, (p-1)q/p; 1)$. 
\item If $q > 0$ divides $p$, then $T_{S^{3}}(p, q; 1) \cong T_{S^{3}}((q-1)p/q, q-1; 2)$. 
\end{enumerate}
The relation (1) holds since the links are not oriented. 
The relation (2) holds by exchanging $H_{0}$ and $H_{1}$. 
The relations (3) and (4) hold by isotoping a component to the core of $H_{0}$ or $H_{1}$. 
Note that $T_{S^{3}}(1, q; 0) \cong T_{S^{3}}(0, 0; 1)$ is the unknot, 
and $T_{S^{3}}(2, 2; 0) \cong T_{S^{3}}(2, -2; 0) \cong T_{S^{3}}(1, 1; 1) \cong T_{S^{3}}(1, -1; 1) \cong T_{S^{3}}(0, 0; 2)$ is the Hopf link. 
In fact, these relations are sufficient to consider the isotopy classes. 

\begin{prop}[\cite{Budney06} Proposition 3.5]
\label{prop:torus-link}
If two links $T_{S^{3}}(p_{0}, q_{0}; n_{0})$ and $T_{S^{3}}(p_{1}, q_{1}; n_{1})$ are isotopic, 
they are related by combining the relations $(1)$--$(4)$. 
\end{prop}

Similarly, the following relations for the links $T_{\bbR \bbP^{3}}(p, q; n)$ hold. 
\begin{enumerate}
\item $T_{\bbR \bbP^{3}}(p, q; n) = T_{\bbR \bbP^{3}}(-p, -q; n)$. 
\item $T_{\bbR \bbP^{3}}(p, q; n) \cong T_{\bbR \bbP^{3}}(-p+2q, q; n)$ for $n=0,2$. 
\item If $p > 0$ divides $q$, then $T_{\bbR \bbP^{3}}(p, q; 0) \cong T_{\bbR \bbP^{3}}(p-1, (p-1)q/p; 1)$. 
\item If $-p+2q > 0$ divides $q$, then $T_{\bbR \bbP^{3}}(p, q; 1) \cong T_{\bbR \bbP^{3}}((-p+2q-1)p/(-p+2q), (-p+2q-1)q/(-p+2q); 2)$. 
\end{enumerate}
To obtain the relation (2), 
we exchange $H_{0}$ and $H_{1}$. 
Then the elements $\mu$ and $2 \lambda + \mu$ are exchanged by a linear involution on $H_{1}(\Sigma, \bbZ) \cong \bbZ^{2}$. 
The equations
$\begin{pmatrix}
a & b \\
c & d 
\end{pmatrix}
\begin{pmatrix}
0 \\
1 
\end{pmatrix}
= 
\begin{pmatrix}
2 \\
1 
\end{pmatrix}$
and 
$\begin{pmatrix}
a & b \\
c & d 
\end{pmatrix}
\begin{pmatrix}
2 \\
1 
\end{pmatrix}
= 
\begin{pmatrix}
0 \\
1 
\end{pmatrix}$ 
imply 
$\begin{pmatrix}
a & b \\
c & d 
\end{pmatrix}
=
\begin{pmatrix}
-1 & 2 \\
0 & 1 
\end{pmatrix}$. 
Hence the elements $p \lambda + q \mu$ and $(-p+2q) \lambda + q \mu$ are exchanged. 
We will show that the relations (1)--(4) are sufficient in Corollary~\ref{cor:rp3-torus-link}.

We next prepare notations for Seifert fibrations. 
For coprime integers $a>0$ and $b$, 
the \emph{$(a,b)$-fibered solid torus} is obtained from $D^{2} \times [0,1]$ by identifying $D^{2} \times \{ 0 \}$ and $D^{2} \times \{ 1 \}$ through the $2\pi b/a$-rotation, 
and its fibers are obtained from the fibers $\{ z \} \times [0,1]$ of $D^{2} \times [0,1]$. 
A $(1,0)$-fibered solid torus is also called a \emph{product fibered solid torus}. 
A \emph{Seifert fibration} of a 3-manifold is a decomposition into a disjoint union of circles (called \emph{fibers}) such that each fiber has a tubular neighborhood that forms a fibered solid torus. 
A fiber is called \emph{regular} if it has a tubular neighborhood that forms a product fibered solid torus. 
A fiber is called \emph{$(a,b)$-singular} if it has a tubular neighborhood that forms an $(a,b)$-fibered solid torus for $a \geq 2$. 
The following fact is well known. 

\begin{lem}
\label{lem:extension}
Let $M$ be a Seifert fibered 3-manifold with boundary consisting of tori. 
Let $T$ be a component of the boundary of $M$, which is a union of fibers. 
Suppose that a 3-manifold $M'$ is obtained by attaching a solid torus to $M$ along $T$. 
Then the Seifert fibration of $M$ is uniquely extended to a Seifert fibration of $M'$ by a fibered solid torus 
unless a meridian (i.e., a simple closed curve on $T$ bounding a disk in the solid torus) is isotopic to a fiber in $T$. 
\end{lem}

The Seifert fibrations are classified (see \cite{Hatcher07} for details). 
Every Seifert fibration of $S^{3}$ or $\bbR \bbP^{3}$ is obtained by gluing two fibered solid tori, 
and so it is compatible with the genus-1 Heegaard splitting.

We now consider the links $T_{*}(p, q; n)$ in relation to Seifert fibrations. 
For $(p, q) \in \bbZ^{2} \setminus \{ (0,0) \}$, 
take a Seifert fibration of a tubular neighborhood of the Heegaard torus $\Sigma$ 
so that a link $T_{*}(p, q; 0)$ consists of fibers. 
This Seifert fibration has no singular fibers. 
By Lemma~\ref{lem:extension}, this can be extended to a Seifert fibration of $S^{3}$ or $\bbR \bbP^{3}$ 
if the slope $p' \lambda + q' \mu$ is neither the meridian of $H_{0}$ nor $H_{1}$. 
Then the cores of $H_{0}$ and $H_{1}$ are (in general singular) fibers. 
Hence a link $T_{*}(p, q; n)$ consists of fibers of a Seifert fibration 
unless it is isotopic to $T_{*}(0, q; n)$ or $T_{\bbR \bbP^{3}}(2q, q; 1)$. 
Note that $T_{S^{3}}(p, 0; 1) \cong T_{S^{3}}(p+1, 0; 0) \cong T_{S^{3}}(0, p+1; 0)$ for $p \geq 0$. 
The knot $T_{S^{3}}(0, 0; 1) \cong T_{S^{3}}(0, 1; 0)$ is the unknot, and so its complement is a solid torus. 
The complement of the knot $T_{\bbR \bbP^{3}}(0, 0; 1) \cong T_{\bbR \bbP^{3}}(1, q; 0)$ is also a solid torus. 
The other links $T_{*}(0, q; 0)$ ($q \geq 2$) and $T_{\bbR \bbP^{3}}(2q, q; 1)$ ($q \geq 1$) are not irreducible, 
and so they are not Seifert fibered. 
Note that a Seifert fibered 3-manifold is irreducible 
unless it is $S^{1} \times S^{2}$, the twisted $S^{2}$-bundle over $S^{1}$, or $\bbR \bbP^{3} \# \bbR \bbP^{3}$ \cite[Proposition 1.12]{Hatcher07}. 
The link $T_{S^{3}}(0, q; 1) \cong T_{S^{3}}(0, q-1; 2)$ for $q \geq 2$ is called a ``key-chain link'' in \cite{Budney06}. 
The following homeomorphisms hold. 
\begin{itemize}
\item The complement of $T_{S^{3}}(0, q; 1)$ is homeomorphic to the complement of $q$ regular fibers in the product fibered solid torus, 
which is the product of the $q$-punctured disk and the circle. 
\item The complement of $T_{\bbR \bbP^{3}}(0, q; 1)$ is homeomorphic to the complement of $q$ regular fibers in the $(2,1)$-fibered solid torus. 
\item The complement of $T_{\bbR \bbP^{3}}(0, q; 2)$ is homeomorphic to the complement of $q+1$ regular fibers in the product fibered solid torus, 
which is the product of the $(q+1)$-punctured disk and the circle. 
\end{itemize}
These are Seifert fibered. 
Hence the complement of $T_{*}(p, q; n)$ is Seifert fibered 
unless it is isotopic to $T_{*}(0, q; 0)$ ($q \geq 2$) or $T_{\bbR \bbP^{3}}(2q, q; 1)$ ($q \geq 1$). 
The converse is also true. 

\begin{prop}[\cite{Budney06} Proposition 3.3]
\label{prop:s3-seifert}
Let $L$ be a link in $S^{3}$. 
Suppose that the complement of $L$ is Seifert fibered. 
Then $L$ is isotopic to a link $T_{S^{3}}(p, q; n)$. 
\end{prop}

\begin{prop}
\label{prop:rp3-seifert}
Let $L$ be a link in $\bbR \bbP^{3}$. 
Suppose that the complement of $L$ is Seifert fibered. 
Then $L$ is isotopic to a link $T_{\bbR \bbP^{3}}(p, q; n)$. 
\end{prop}
\begin{proof}
We first suppose that $L$ consists of fibers of a Seifert fibration of $\bbR \bbP^{3}$. 
A Seifert fibration of $\bbR \bbP^{3}$ is obtained by gluing two fibered solid tori. 
These solid tori induce a Heegaard splitting of $\bbR \bbP^{3}$. 
The singular fibers are the cores of the solid tori. 
The components of $L$ which are regular fibers can be isotoped into the Heegaard torus. 
Hence $L$ is isotopic to a link $T_{\bbR \bbP^{3}}(p, q; n)$. 

We next suppose that $L$ does not consist of fibers of any Seifert fibration of $\bbR \bbP^{3}$. 
Fix a Seifert fibration $\mathcal{F}$ on the complement of $L$. 
Then Lemma~\ref{lem:extension} implies that 
there is a component $K_{0}$ of $L$ whose meridian is a fiber. 
Let $\mathcal{S}$ denote the base orbifold of the Seifert fibration $\mathcal{F}$. 
Consider a surface in the complement of $L$ that is the preimage of an embedded circle on $\mathcal{S}$. 
Due to Bredon and Wood \cite{BW69} (see also \cite{GT22}), 
there exist no embedded Klein bottles in $\bbR \bbP^{3}$. 
Since there exist no surjections from $\pi_{1}(\bbR \bbP^{3}) \cong \bbZ / 2\bbZ$ to $\pi_{1}(S^{1}) \cong \bbZ$, 
there exist no non-separating tori in $\bbR \bbP^{3}$. 
Hence the orbifold $\mathcal{S}$ is orientable and has zero genus. 

Suppose that there is a singular fiber of $\mathcal{F}$. 
Let $\partial_{0} \mathcal{S}$ denote the boundary component of $\mathcal{S}$ that is the image of the boundary of the tubular neighborhood $N(K_{0})$ of $K_{0}$. 
Let $\alpha$ be an embedded arc in $\mathcal{S}$ that joins a singular point of $\mathcal{S}$ and a point in $\partial_{0} \mathcal{S}$. 
Let $\beta$ be an embedded arc in $\mathcal{S}$ that joins two points in $\partial_{0} \mathcal{S}$
and is contained in the boundary of a tubular neighborhood of $\alpha$. 
Consider a sphere in $\bbR \bbP^{3}$ that is a union of the preimage of $\beta$ and two disks in $N(K_{0})$. 
If this sphere surrounds an $(a,b)$-singular fiber, 
it induces a connected sum decomposition of $\bbR \bbP^{3}$ one of whose summands is a lens space $L(a,b')$. 
Since $\bbR \bbP^{3} = L(2,1)$ is irreducible, 
the Seifert fibration $\mathcal{F}$ has at most one singular fiber, 
which is a $(2,1)$-singular fiber. 

Let $K_{1}, \dots, K_{m}$ denote the components of $L$ other than $K_{0}$. 
Consider $m$ embedded arcs in $\mathcal{S}$ around the corresponding boundary components that join two points in $\partial_{0} \mathcal{S}$. 
By the same argument as above, we can determine the link $L$ as follows. 
\begin{itemize}
\item 
If $\mathcal{F}$ has a $(2,1)$-singular fiber, 
each of the meridians of $K_{1}, \dots, K_{m}$ intersects a regular fiber at one point. 
Then $L$ is isotopic to a link $T_{\bbR \bbP^{3}}(0, q; 1)$. 
\item 
If $\mathcal{F}$ has no singular fibers, 
exactly one of the meridians of $K_{1}, \dots, K_{m}$ intersects a regular fiber at two points, 
and each of the others intersects a regular fiber at one point. 
Then $L$ is isotopic to a link $T_{\bbR \bbP^{3}}(0, q; 2)$. 
\end{itemize}
\end{proof}

\begin{lem}
\label{lem:t-link-lift}
The preimage of the link $T_{\bbR \bbP^{3}}(p, q; n)$ by the covering map $\pi \colon S^{3} \to \bbR \bbP^{3}$ 
is the link $T_{S^{3}}(p, -p+2q; n)$. 
\end{lem}
\begin{proof}
To avoid confusion, we put tildes on the symbols for $S^{3}$, 
such as $\tilde{\lambda}$ and $\tilde{\mu}$. 
We may assume that the covering map $\pi$ is compatible with the Heegarrd splittings; 
each $\widetilde{H}_{i}$ in $S^{3}$ is mapped to $H_{i}$ in $\bbR \bbP^{3}$. 
Then the preimage of $T_{\bbR \bbP^{3}}(p, q; n)$ is $T_{S^{3}}(\tilde{p}, \tilde{q}; n)$ for some $\tilde{p}$ and $\tilde{q}$. 
The preimage of $p \lambda + q \mu$ is $p \tilde{\lambda}' + 2q \tilde{\mu}$, 
where $\tilde{\lambda}'$ and $2 \tilde{\mu}$ are the preimages of $\lambda$ and $\mu$, respectively. 
Since $2 \lambda + \mu$ is the meridian of $H_{1}$ 
and lifted to twice of the longitude $\tilde{\lambda}$ of $\widetilde{H}_{0}$, 
we have $2\tilde{\lambda}' + 2\tilde{\mu} = 2\tilde{\lambda}$. 
Hence $p \tilde{\lambda}' + 2q \tilde{\mu} = p \tilde{\lambda} + (-p+2q) \tilde{\mu}$, 
and so $(\tilde{p}, \tilde{q}) = (p, -p+2q)$. 
\end{proof}

\begin{thm}
\label{thm:seifert-main}
Let $L_{0}$ and $L_{1}$ be links in $\bbR \bbP^{3}$ whose complements are Seifert fibered. 
Suppose that the links $\widetilde{L_{0}} = \pi^{-1}(L_{0})$ and $\widetilde{L_{1}} = \pi^{-1}(L_{1})$ in $S^{3}$ are isotopic. 
Then $L_{0}$ and $L_{1}$ are isotopic. 
\end{thm}
\begin{proof}
By Proposition~\ref{prop:rp3-seifert}, 
we may assume that $L_{0} = T_{\bbR \bbP^{3}}(p_{0}, q_{0}; n_{0})$ and $L_{1} = T_{\bbR \bbP^{3}}(p_{1}, q_{1}; n_{1})$. 
By Lemma~\ref{lem:t-link-lift}, 
we have $\widetilde{L_{0}} = T_{S^{3}}(p_{0}, -p_{0} + 2q_{0}; n_{0})$ and $\widetilde{L_{1}} = T_{S^{3}}(p_{1}, -p_{1}+2q_{1}; n_{1})$, which are isotopic. 
We need to check the relations (1)--(4) in Proposition~\ref{prop:torus-link}. 
In general cases for the relations (1) and (2), 
$n_{0} = n_{1}$ and $(p_{0}, -p_{0}+2q_{0})$ is equal to $\pm (p_{1}, -p_{1}+2q_{1})$ or $\pm (-p_{1}+2q_{1}, p_{1})$. 
Then $(p_{0}, q_{0})$ is equal to $\pm (p_{1}, q_{1})$ or $\pm (-p_{1}+2q_{1}, q_{1})$. 
Hence $L_{0}$ and $L_{1}$ are isotopic. 

To complete the proof, 
it is sufficient to consider the following exceptional cases. 
In the last case, the relations (3) and (4) are combined. 
\begin{itemize}
\item Suppose that $p_{0} > 0$ divides $-p_{0}+2q_{0}$, $n_{0} = 0$, $n_{1} = 1$, $p_{1} = p_{0}-1$, and $-p_{1}+2q_{1} = (p_{0}-1)(-p_{0}+2q_{0})/p_{0}$. 
Note that $T_{S^{3}}(p_{0}, -p_{0}+2q_{0}; 1) \cong T_{S^{3}}(p_{0}-1, (p_{0}-1)(-p_{0}+2q_{0})/p_{0}; 0)$. 
Then $q_{1} = (p_{0}-1)q_{0}/p_{0}$ and $p_{0}$ divides $q_{0}$. 
Hence $L_{0} = T_{\bbR \bbP^{3}}(p_{0}, q_{0}; 0)$ and $L_{1} = T_{\bbR \bbP^{3}}(p_{0}-1, (p_{0}-1)q_{0}/p_{0}; 1)$ are isotopic. 
\item Suppose that $-p_{0}+2q_{0} > 0$ divides $p_{0}$, $n_{0} = 1$, $n_{1} = 2$, $p_{1} = (-p_{0}+2q_{0}-1)p_{0}/(-p_{0}+2q_{0})$, and $-p_{1}+2q_{1} = -p_{0}+2q_{0}-1$. 
Note that $T_{S^{3}}(p_{0}, -p_{0}+2q_{0}; 2) \cong T_{S^{3}}((-p_{0}+2q_{0}-1)p_{0}/(-p_{0}+2q_{0}), -p_{0}+2q_{0}-1; 1)$. 
Then $q_{1} = (-p_{0}+2q_{0}-1)q_{0}/(-p_{0}+2q_{0})$ and $-p_{0}+2q_{0}$ divides $q_{0}$. 
Hence $L_{0} = T_{\bbR \bbP^{3}}(p_{0}, q_{0}; 1)$ and $L_{1} = T_{\bbR \bbP^{3}}((-p_{0}+2q_{0}-1)p_{0}/(-p_{0}+2q_{0}), (-p_{0}+2q_{0}-1)q_{0}/(-p_{0}+2q_{0}); 2)$ are isotopic. 
\item Suppose that $p_{0} = \pm (-p_{0}+2q_{0}) > 0$, $n_{0} = 0$, $n_{1} = 2$, 
$p_{1} = p_{0}-2$, and $-p_{1}+2q_{1} = \pm (p_{0}-2)$. 
Note that $T_{S^{3}}(p_{0}, -p_{0}+2q_{0}; 0) = T_{S^{3}}(p_{0}, \pm p_{0}; 2) \cong T_{S^{3}}(p_{0}-2, \pm (p_{0}-2); 2)$. 
If $p_{0} = -p_{0}+2q_{0}$, then $q_{0} = p_{0}$ and $q_{1} =  p_{0}-2$. 
Hence $L_{0} = T_{\bbR \bbP^{3}}(p_{0}, q_{0}; 0)$ and $L_{1} = T_{\bbR \bbP^{3}}(p_{0}-2, p_{0}-2; 2)$ are isotopic. 
If $p_{0} = -(-p_{0}+2q_{0})$, then $q_{0} =  q_{1} =  0$. 
Hence $L_{0} = T_{\bbR \bbP^{3}}(p_{0}, q_{0}; 0)$ and $L_{1} = T_{\bbR \bbP^{3}}(p_{0}-2, 0; 2)$ are isotopic. 
\end{itemize}
\end{proof}

We finally remark that the relations (1)--(4) for the isotopy classes of $T_{\bbR \bbP^{3}}(p, q; n)$ are sufficient, 
for otherwise more preimages in $S^{3}$ would be identified. 
Hence: 

\begin{cor}
\label{cor:rp3-torus-link}
If two links $T_{\bbR \bbP^{3}}(p_{0}, q_{0}; n_{0})$ and $T_{\bbR \bbP^{3}}(p_{1}, q_{1}; n_{1})$ are isotopic, 
they are related by combining the relations $(1)$--$(4)$. 
\end{cor}

\section{Prime and JSJ decompositions}
\label{section:jsj}

In this section, we describe the prime and JSJ decompositions of the complement of a link in $\bbR \bbP^{3}$, 
and we prove the main theorem. 
See \cite{Hatcher07} for details of the prime and JSJ decompositions.

The prime decomposition \cite{Kneser29, Milnor62} of a compact orientable 3-manifold is the unique maximal decomposition by the connected sum. 
Let $L$ be a link in a compact orientable irreducible 3-manifold $X$. 
The link $L$ is called a \emph{split} link if the complement of $L$ is reducible. 
We say that $L$ splits to links $L', L_{1}, \dots, L_{m}$ 
if $L = L' \sqcup L_{1} \sqcup \dots \sqcup L_{m}$ and there are mutually disjoint balls $B_{1}, \dots B_{m}$ in $X$ disjoint from $L'$ such that each $B_{i}$ contains $L_{i}$. 
Then a split summand $L_{i}$ can be regarded as a link in $S^{3}$. 
If $X \neq S^{3}$ and $m \geq 1$, then $L$ is a split link. 
The prime decomposition implies the following lemma. 

\begin{lem}[\cite{KMMY25} Lemma 2.9]
\label{lem:split}
Let $L$ be a link in an irreducible compact orientable 3-manifold $X$. 
Then $L$ maximally splits to links $L^{\prime}$, $L_{1}$, \dots, $L_{m}$, 
where $L^{\prime}$ is a (possibly empty) non-split link, and $L_{1}$, \dots, $L_{m}$ are contained in disjoint 3-balls. 
Moreover, this splitting is unique up to permutation of $L_{1}$, \dots, $L_{m}$. 
\end{lem}

The JSJ decomposition \cite{JS79, Johannson79} of a compact orientable irreducible 3-manifold with toral boundary is the unique minimal decomposition along mutually disjoint (possibly empty) incompressible tori (called the \emph{JSJ tori}) such that 
each piece after decomposition is Seifert fibered or (algebraically) atoroidal. 
Thurston's geometrization~\cite{Thurston82} implies that an irreducible, boundary-irreducible, atoroidal 3-manifold with non-empty toral boundary admits a hyperbolic metric of finite volume. 
Consequently, the complement of a non-split link in $X$ admits a unique decomposition along tori into pieces each of which is Seifert fibered or hyperbolic. 

To consider the JSJ decomposition of the complement of a link in $\bbR \bbP^{3}$, 
we need to describe a torus in $\bbR \bbP^{3}$. 
We call a submanifold $Y$ of a 3-manifold $X$ that satisfies the following conditions a \emph{knotted hole ball}. 
\begin{itemize}
\item There is a closed 3-ball $B$ contained in $X$. 
\item There is an embedding of a 1-handle $\iota \colon [0,1] \times D^{2} \to B$. 
\item $\iota^{-1}(\partial B) = \{ 0,1 \} \times D^{2}$. 
\item $Y = \overline{B \setminus \iota ([0,1] \times D^{2})}$. 
\item $Y$ is not a solid torus. 
\end{itemize}
Note that the complement of a knotted hole ball in $X$ is the connected sum of $X$ and $S^{1} \times D^{2}$.

\begin{lem}
\label{lem:torus}
An embedded torus in $\bbR \bbP^{3}$ bounds a solid torus or a knotted hole ball. 
\end{lem}
\begin{proof}
Since $\pi_{1}(\bbR \bbP^{3}) \cong \bbZ / 2\bbZ$ is finite, 
any torus $T$ in $\bbR \bbP^{3}$ is compressible. 
A sphere $S$ is obtained by compressing $T$ along a compressing disk $D$. 
The sphere $S$ bounds a ball $B$ since $\bbR \bbP^{3}$ is irreducible. 
If $B$ and $D$ are disjoint, then $T$ bounds a solid torus. 
Otherwise, $T$ bounds the complement of a 1-handle in $B$, which is a solid torus or a knotted hole ball. 
\end{proof}

Let $L$ be a non-split link in $\bbR \bbP^{3}$. 
We say that a JSJ piece $M$ of the complement of $L$ is \emph{outermost} 
if $\overline{\bbR \bbP^{3} \setminus M}$ is a disjoint union of solid tori and knotted hole balls. 
This definition can be applied to a non-split link in $S^{3}$ or $S^{1} \times D^{2}$. 
However, every JSJ piece for a non-split link in $S^{3}$ is outermost. 
The outermost piece for a non-split link in $S^{1} \times D^{2}$ contains the boundary of $S^{1} \times D^{2}$, and so it is unique. 
In any case, an outermost piece can be re-embedded as a link complement 
by reattaching the 1-handles adjacent to the knotted hole balls 
as in \cite[Proposition 2.2]{Budney06} and \cite[Lemma 3.6]{KMMY25}. 

\begin{lem}
\label{lem:outermost}
The complement of a link $L$ in $\bbR \bbP^{3}$ has at least one outermost JSJ piece. 
\end{lem}
\begin{proof}
The JSJ graph $G_{L} = (V_{L}, E_{L})$ for the complement of $L$ is defined as follows. 
\begin{itemize}
\item Each vertex in the vertex set $V_{L}$ corresponds to a JSJ piece. 
\item Each edge in the edge set $E_{L}$ corresponds to a JSJ torus. 
\item The endpoints of an edge correspond to the adjacent pieces about the corresponding JSJ torus. 
\end{itemize}
Since any torus in $\bbR \bbP^{3}$ is separating, the graph $G_{L}$ is a tree. 
We give $G_{L}$ the structure of a partially directed graph as follows. 
Let $e$ be an edge corresponding to a JSJ torus $T$. 
\begin{itemize}
\item If $T$ bounds a knotted hole ball, then $e$ has the orientation to the piece contained in the knotted hole ball. 
\item If $T$ bounds exactly one solid torus, then $e$ has the orientation to the piece contained in the solid torus. 
\item If $T$ bounds two solid tori (i.e., $T$ is a Heegaard torus), then $e$ is unoriented. 
\end{itemize}
Note that $T$ cannot bound a knotted hole ball to one side and a solid torus to the other side, for otherwise the ambient space would be $S^{3}$. 
Hence the orientation is well defined. 

Suppose that a function $f \colon V_{L} \to \bbZ$ satisfies the following. 
Let $v_{0}$ and $v_{1}$ be the endpoints of an edge $e$. 
\begin{itemize}
\item If $e$ is oriented from $v_{0}$ to $v_{1}$, then $f(v_{0}) + 1 = f(v_{1})$. 
\item If $e$ is unoriented, then $f(v_{0}) = f(v_{1})$. 
\end{itemize}
Since $G_{L}$ is a tree, such a function $f$ exists. 
Suppose that a vertex $v \in V_{L}$ attains a local minimum of $f$. 
Then each edge incident to $v$ is oriented from $v$ or unoriented. 
Let $M$ denote the JSJ piece corresponding to $v$. 
Each boundary component of $M$ is a JSJ torus or the boundary of a neighborhood of a component of $L$. 
The orientations of edges imply that each boundary component of $M$ bounds a solid torus or a knotted hole ball disjoint from $M$. 
Hence $M$ is outermost. 
Since $G_{L}$ is finite, the minimum of $f$ exists. 
\end{proof}

For a non-split link $L \subset \bbR \bbP^{3}$, 
let $\widetilde{L} \subset S^{3}$ denote the preimage of $L$ by the covering map $\pi \colon S^{3} \to \bbR \bbP^{3}$. 
By \cite[Lemma 2.11]{KMMY25} and the fact that there exist no embedded Klein bottles in $\bbR \bbP^{3}$ \cite{BW69, GT22}, 
the JSJ tori of the complement of $\widetilde{L}$ is the preimage of the JSJ tori of the complement of $L$ by the covering map $\pi$. 
An outermost JSJ piece for $L$ is characterized as follows. 

\begin{lem}
\label{lem:outermost-lift}
Let $M$ be a JSJ piece of the complement of a link $L$ in $\bbR \bbP^{3}$. 
The piece $M$ is outermost if and only if 
\begin{itemize}
\item the preimage $\widetilde{M} = \pi^{-1}(M)$ is connected, and 
\item the number of components of $\overline{S^{3} \setminus \widetilde{M}}$ that are not solid tori is even. 
\end{itemize}
\end{lem}
\begin{proof}
Since a knotted hole ball is contained in a 3-ball $B$, the preimage of a knotted hole ball in $\bbR \bbP^{3}$ consists of its two copies, which are contained in the preimage of $B$. 
The preimage of a solid torus in $\bbR \bbP^{3}$ is one or two solid tori. 

Suppose that $M$ is outermost. 
In other words, each component of $\overline{\bbR \bbP^{3} \setminus M}$ is a solid torus or a knotted hole ball. 
Assume that $\widetilde{M}$ is disconnected. 
Then the preimage of each component of the boundary $\partial \widetilde{M}$ consists of its two copies. 
Hence the preimage of each solid torus component of $\overline{\bbR \bbP^{3} \setminus M}$ also consists of its two copies. 
Since the preimage of each knotted hole ball must also consist of its two copies, 
the union $S^{3}$ of $\widetilde{M}$, the solid tori, and the knotted hole balls would be disconnected.
Hence $\widetilde{M}$ is connected. 
Since the components of $\overline{S^{3} \setminus \widetilde{M}}$ that are not solid tori 
are the preimage of knotted hole ball components of $\overline{\bbR \bbP^{3} \setminus M}$, 
their number is even. 

Conversely, suppose that $M$ is not outermost. 
Then Lemma~\ref{lem:torus} implies that $M$ is contained in a solid torus or a knotted hole ball. 
If $M$ is contained in a knotted hole ball, 
the preimage $\widetilde{M}$ is disconnected. 
Suppose that there is a JSJ torus $T \subset \partial M$ 
such that $M$ is contained in a solid torus $S$ bounded by $T$. 
Since an embedded torus in a solid torus bounds a solid torus or a knotted hole ball \cite[Lemma 3.3]{KMMY25}, the components of $\overline{S \setminus M}$ consist of solid tori and knotted hole balls. 
If $T$ is a Heegaard torus of $\bbR \bbP^{3}$, then $M$ is outermost. 
If $\widetilde{M}$ is connected, and $T$ is not a Heegaard torus, 
then the components of $\overline{S^{3} \setminus \widetilde{M}}$ that are not solid tori consist of 
\begin{itemize}
\item the preimages of knotted hole balls contained in the solid torus $S$, 
which are pairs, and 
\item the preimage of $\overline{\bbR \bbP^{3} \setminus S}$, 
which is connected, 
\end{itemize}
and so their number is odd. 
Note that if an orientable 3-manifold $N$ is not a solid torus, a finite cover of $N$ cannot be a solid torus. 
\end{proof}

Using the above settings, we prove the main theorem in the same manner as \cite{KMMY25}. 

\begin{proof}[Proof of Theorem~\ref{thm:main}]
The argument for split links is identical to that in the proof of \cite[Theorem 4.5]{KMMY25}. 
Suppose that a link $L_{0}$ maximally splits to $L_{0}', L_{0,1}, \dots, L_{0,m}$, 
where $L_{0}'$ is a (possibly empty) non-split link. 
Then $\widetilde{L_{0}} = \pi^{-1}(L_{0})$ maximally splits to $\widetilde{L_{0}'} = \pi^{-1}(L_{0}')$ and $L_{0,j}^{k}$ for $1 \leq j \leq m$ and $k = 1,2$, 
where $L_{0,j}^{k}$ is a copy of $L_{0,j}$ as a link in $S^{3}$. 
The same applies to $L_{1}$. 
Since $\widetilde{L_{0}}$ and $\widetilde{L_{1}}$ are isotopic, 
their split summands are isotopic by Lemma~\ref{lem:split}. 
Hence we may assume that $L_{0}$ and $L_{1}$ are non-split links. 

We consider the JSJ decompositions of the complements of $L_{0}$, $L_{1}$, and $\widetilde{L_{0}} \cong \widetilde{L_{1}}$. 
As mentioned above, the JSJ decomposition of the complement of $\widetilde{L_{0}} \cong \widetilde{L_{1}}$ is the lifts of those of $L_{0}$ and $L_{1}$. 
By Lemma~\ref{lem:outermost}, there is an outermost JSJ piece $M_{0}$ of the complement of $L_{0}$. 
By Lemma~\ref{lem:outermost-lift}, $\widetilde{M_{0}} = \pi^{-1}(M_{0})$ is a single JSJ piece of the complement of $\widetilde{L_{0}}$. 
There is a JSJ piece $M_{1}$ of the complement of $L_{1}$ such that $\widetilde{M_{1}} = \pi^{-1}(M_{1})$ is isotopic to $\widetilde{M_{0}}$. 
By Lemma~\ref{lem:outermost-lift}, the piece $M_{1}$ is also outermost. 
Since $M_{0}$ and $M_{1}$ are outermost, 
they can be re-embedded as the complements of links in $\bbR \bbP^{3}$. 
Since the preimage of a knotted hole ball in $\bbR \bbP^{3}$ consists of its two copies in $S^{3}$, 
the pieces $\widetilde{M_{0}}$ and $\widetilde{M_{1}}$ are equivariantly re-embedded as link complements. 
Since $M_{0}$ and $M_{1}$ are hyperbolic or Seifert fibered, 
they are isotopic in $\bbR \bbP^{3}$ by Theorems \ref{thm:hyp-main} and \ref{thm:seifert-main}. 

We show that the isotopy between $M_{0}$ and $M_{1}$ extends to an isotopy between the complements of $L_{0}$ and $L_{1}$ 
using an isotopy between the complements of $\widetilde{L_{0}}$ and $\widetilde{L_{1}}$. 
Since $M_{0}$ and $M_{1}$ are outermost, 
the remaining regions consist of solid tori and knotted hole balls. 
Since the preimage of a knotted hole ball in $\bbR \bbP^{3}$ consists of its two copies in $S^{3}$, 
the isotopy extends to the knotted hole balls. 
By \cite[Theorem 4.5]{KMMY25} for links in the solid torus, 
the isotopy in $S^{3}$ induces an isotopy between links in the remaining solid tori contained in $\bbR \bbP^{3}$. 
Therefore $L_{0}$ and $L_{1}$ are isotopic. 
\end{proof}

\section*{Acknowledgements} 
The author is grateful to Yuya Koda, Yuka Kotorii, Sonia Mahmoudi, Elisabetta A. Matsumoto, and Yuta Nozaki for their helpful discussions. 
This work was supported by the World Premier International Research Center Initiative Program, International Institute for Sustainability with Knotted Chiral Meta Matter (WPI-SKCM$^2$), MEXT, Japan, 
and JSPS Program for Forming Japan's Peak Research Universities (J-PEAKS) Grant Number JPJS00420230011. 

\bibliographystyle{plain}
\bibliography{ref-uf2p}

\end{document}